\documentclass[11pt,oneside,english,british]{amsart}
\usepackage[T1]{fontenc}
\usepackage[latin9]{inputenc}
\usepackage{amsthm}
\usepackage{graphicx}
\usepackage{amssymb}
\usepackage{esint}

\makeatletter

\let\SF@@footnote\footnote
\def\footnote{\ifx\protect\@typeset@protect
    \expandafter\SF@@footnote
  \else
    \expandafter\SF@gobble@opt
  \fi
}
\expandafter\def\csname SF@gobble@opt \endcsname{\@ifnextchar[%]
  \SF@gobble@twobracket
  \@gobble
}
\edef\SF@gobble@opt{\noexpand\protect
  \expandafter\noexpand\csname SF@gobble@opt \endcsname}
\def\SF@gobble@twobracket[#1]#2{}

\numberwithin{equation}{section} %% Comment out for sequentially-numbered
\numberwithin{figure}{section} %% Comment out for sequentially-numbered
\theoremstyle{plain}
\theoremstyle{plain}
\newtheorem{thm}{Theorem}
  \theoremstyle{definition}
  \newtheorem{defn}[thm]{Definition}
  \theoremstyle{plain}
  \newtheorem{lem}[thm]{Lemma}
 \theoremstyle{definition}
  \newtheorem{example}[thm]{Example}

\usepackage[T1]{fontenc}

\makeatother

\usepackage{babel}

\begin{document}

\title[Generalised W.T.N \& space of Lagrangian submersions]{Generalized Weinstein Tubular neighbourhood \&\\
Space of Lagrangian submersions}

\author{Nicolas Roy}

\date{30 April 2009}

\selectlanguage{english}%

\address{Humbolt universität zu Berlin, Germany. }

\email{\texttt{roy@math.hu-berlin.de}}

\selectlanguage{british}%

\dedicatory{This article is dedicated to Kat.}

\maketitle

\section{Introduction}

The standard Weinstein tubular neighbourhood (WTN in the sequel) theorem
\cite{weinstein1} states that for any Lagrangian submanifold $L_{0}$
of a symplectic manifold $M$, there is a tubular neighbourhood of
$L_{0}$ which is symplectomorphic to a tubular neighbourhood of the
$0$-section in the cotangent bundle $T^{*}L_{0}$. Through this symplectomorphism
the Lagrangian submanifolds close to $L_{0}$ correspond therefore
to Lagrangian sections of $T^{*}L_{0}$, i.e., closed $1$-forms on
$L_{0}$. This allows to give a (infinite dimensional) manifold structure
to the space $\mathcal{L}\left(M\right)$ of all Lagrangian submanifolds
of $M$, as a submanifold of the space $\mathcal{S}\left(M\right)$
of all submanifolds.

The aim of this paper is to present a generalisation of the standard
WTN, where $M$ is replaced by a symplectic bundle%
\footnote{A symplectic bundle is a locally trivial fibration $p:X\rightarrow B$
where the fibres of $p$ are equipped with a symplectic form, which
depends smoothly on the point on the base $B$. See Section \ref{sub-symplectic-bundles}.%
} $p:X\rightarrow B$ and $L$ by a subbundle $p_{L}:L\rightarrow B$,
with $L\subset X$, whose fibres are Lagrangian in the fibres of $p$.
This allows us to prove that the space $\mbox{Subb}_{\mbox{Lag}}\left(X\right)$
of Lagrangian subbundles of $X$ is a Fréchet submanifold of the space
$\mbox{Subb}\left(X\right)$ of all subbundles.

Our main motivation is actually the study of the space $\mbox{Fib}_{\mbox{Lag}}\left(M,B\right)$
of Lagrangian fibrations $\pi:M\rightarrow B$ of a given symplectic
manifold $M$. This fits in the previous setting if one considers
the symplectic bundle $X=M\times B$ over $B$. Indeed, the graph
of the section $\left(\mathbb{I},\pi\right)\in\Gamma\left(M,X\right)$
is a Lagrangian subbundle%
\footnote{Here $X$ is viewed alternatively as a bundle over $M$ or over $B$. %
} of $X$. This provides a way to show that $\mbox{Fib}_{\mbox{Lag}}\left(M,B\right)$
is a Fréchet manifold. This is even a principal bundle with structure
group $\mbox{Diff}\left(B\right)$, but this will be explained in
subsequent papers \cite{humiliere_roy_1,humiliere_roy_2}.

Our interest in the space of Lagrangian fibrations comes from completely
integrable systems. Indeed, a certain class of completely integrable
Hamiltonian systems on a given symplectic manifold $M$ is well-parametrised
by the quotient of $C^{\infty}\left(B\right)\times\mbox{Fib}_{\mbox{Lag}}\left(M,B\right)$
by $\mbox{Diff}\left(B\right)$, which is a Fréchet manifold, as shown
in \cite{humiliere_roy_1}.

The plan of the present paper is the following. First, we introduce
the notion of fibred differential forms. Roughly speaking, given a
fibration $p:X\rightarrow B$, a $p$-fibred form is a smooth differential
form on $M$ but which can be contracted only with vectors tangent
to the fibres of $p$. It turns out that the space of $p$-fibred
$1$-forms $\Omega_{p}^{1}\left(X\right)$ is precisely the suitable
model space which describes the neighbourhood in $\mbox{Subb}\left(X\right)$
of a given Lagrangian subbundle $L\in\mbox{Subb}_{\mbox{Lag}}\left(X\right)$,
in analogy with the space $\Omega^{1}\left(T^{*}N\right)$ which is
the model space for a neighbourhood in $\mathcal{S}\left(M\right)$
of a given Lagrangian submanifold $L\in\mathcal{L}\left(X\right)$.
In Section \ref{sec:fibred-differential-forms}, we review the main
properties of these fibred forms. In particular, we explain the notion
of fibred differential $d_{p}$ and show that the space of fibred
closed $1$-forms $\mathcal{Z}_{p}^{1}\left(X\right)$ is a direct
summand%
\footnote{This means that $\mathcal{Z}_{p}^{1}\left(X\right)$ is closed and
$\Omega_{p}^{1}\left(X\right)$ is (topologically) isomorphic to $\mathcal{Z}_{p}^{1}\left(X\right)\oplus F$,
where $F\subset\Omega_{p}^{1}\left(X\right)$ is a closed subspace.%
} of $\Omega_{p}^{1}\left(X\right)$. 

Section \ref{sec-WTN-symplectic-bundles} is devoted to the WTN for
Lagrangian subbundles which is proved in Theorem \ref{thm-WTN-for-lag-subbundle}.
The idea of the construction is the following. A Lagrangian subbundle
$L\subset X$ of a symplectic bundle $p:X\rightarrow B$ can be viewed
as a smooth family $\left(L_{b}\right)_{b\in B}$ of Lagrangian submanifolds
of the symplectic manifolds $X_{b}$. The idea is to construct a family
of (usual) Weinstein tubular neighbourhoods $U_{b}\rightarrow T^{*}L_{b}$
of $L_{b}$. To insure the smoothness with respect to $b\in B$ of
such a family of WTN, one needs a way to fix a particular WTN for
each $b$, since they are not unique. This is done by observing that
for each $b,$ the Weinstein tubular neighbourhoods $U_{b}\rightarrow T^{*}L_{b}$
are in $1:1$ correspondence with Lagrangian foliations transverse
to $L_{b}$, which are themselves in $1:1$ correspondence with Liouville
forms vanishing on $L_{b}$. The corresponding results for symplectic
bundles are explained in Lemma \ref{lem-exist-transverse-pola} and
Lemma \ref{lem-existe-liouville-form}. Then, the fibred WTN Theorem
is used to prove in Theorem \ref{thm-lag-subbundle-sous-varioche}
that $\mbox{Subb}_{\mbox{Lag}}\left(X\right)$ is a Fréchet submanifold
of $\mbox{Subb}\left(X\right)$. Finally, Section \ref{sec-lag-fibrations}
is devoted to the precise description of the relation between Lagrangian
subbundles of symplectic bundles on the one hand and Lagrangian fibrations
of symplectic manifolds on the other hand. The last result of this
article is Theorem \ref{thm-fib_lag-submanifold-of-fib} which proves
that $\mbox{Fib}_{\mbox{Lag}}\left(M,B\right)$ is a Fréchet submanifold
of $\mbox{Fib}\left(M,B\right)$.

\section{\label{sec:fibred-differential-forms}fibred differential forms}

\subsection{Definitions}

We give here the formal definitions which allows us to speak about
smooth families of objects living on different manifolds. The key
idea is of course to view the parameter space of the families as the
base space of a fibration. Throughout this section we fix smooth connected
manifolds $X$ and $B$, the latter being called the base, and a fibration
(i.e., a locally trivial fibre bundle structure) $p:X\rightarrow B$. 

The collection of $V_{p}\left(x\right)=\ker Dp\left(x\right)$ for
all $x\in X$ defines a vector subbundle of $TX$ over $X$, which
we call the \textbf{$p$-vertical tangent bundle} and denote by $V_{p}$. 
\begin{defn}
\label{def-vertical-vector-fields}The sections of $V_{p}$ are called
\textbf{vertical vector fields }and the space of these sections is
denoted by $\frak{X}_{p}\left(X\right)$.
\end{defn}
Then, for any integer $k$, we define the bundle $\bigwedge^{k}V_{p}^{*}$
of totally antisymmetric $k$-linear forms on $V_{p}$, which is also
a vector bundle over $X$. In particular $V_{p}^{*}$ is called the
\textbf{$p$-vertical cotangent bundle}. 
\begin{defn}
\label{def-fibred-forms}The space of smooth sections $\Omega_{p}^{k}\left(X\right)=\Gamma\left(X,\bigwedge^{k}V_{p}^{*}\right)$
is called the space of \textbf{$p$-fibred $k$-forms}.
\end{defn}
Notice that given such a fibred $k$-form $\alpha\in\Omega_{p}^{k}\left(X\right)$,
its restriction $\alpha_{b}$ to any of the fibres $p^{-1}\left(b\right)$,
$b\in B$, is a $k$-form in the usual sense on this fibre. 

In any local trivialisation of $p$, such an $\alpha$ appears as
a smooth family of differential forms on the typical fibre (see below).
fibred differential forms are therefore the right notion that makes
precise the intuitive idea of a smooth family of differential forms
not living on the same manifold.

In the sequel it will be convenient to denote by $X_{b}$ the fibre
$p^{-1}\left(b\right)$ and $\alpha_{b}\in\Omega^{k}\left(X_{b}\right)$
the restriction to $X_{b}$ of a fibred $k$-form $\alpha$.

\subsection{\label{sub-equivalence-with-families}Equivalence with families of
differential forms}

As mentioned above, to any fibred form $\alpha\in\Omega_{p}^{k}\left(X\right)$
we associate the family of restrictions $\left(\alpha_{b}\right)_{b\in B}$,
with $\alpha_{b}\in\Omega^{k}\left(X_{b}\right)$. It is natural to
interpret it as a section of the vector bundle \[
W_{p}^{k}\left(X\right)=\bigcup_{b\in B}\Omega^{k}\left(X_{b}\right).\]
This space is indeed a Fréchet vector bundle over the finite dimensional
manifold $B$. This is shown by finding local trivialisations. For
this, we use a local trivialisation of the bundle $X$, i.e., diffeomorphisms
$\phi:O\times X_{b_{0}}\rightarrow p^{-1}\left(O\right)$, with $O\subset B$
an open set and $b_{0}\in O$, which sends the fibre $\left\{ b\right\} \times X_{b_{0}}$
to the fibre $p^{-1}\left(b\right)$. This gives a family $\phi_{b}=\phi\left(b,.\right)$
of diffeomorphisms from $X_{b_{0}}$ to $X_{b}$, which provides through
$\phi_{b}^{*}:\Omega^{k}\left(X_{b}\right)\rightarrow\Omega^{k}\left(X_{b_{0}}\right)$
a local trivialisation of the vector bundle $W_{p}^{k}\left(X\right)$.
Moreover, the family $\left(\alpha_{b}\right)_{b\in B}$ is a smooth
section of this bundle. This is the way one can interpret any fibred
form $\alpha\in\Omega_{p}^{k}\left(X\right)$ as a smooth family of
forms on the fibres $X_{b}$.

\subsection{Fibred de Rham differential }

For any fibred form $\alpha\in\Omega_{p}^{k}\left(X\right)$ we can
define its exterior derivative $d_{p}\alpha$ which is an element
of $\Omega_{p}^{k+1}\left(X\right)$ defined by the familiar-looking
formula \begin{eqnarray*}
\left(d_{p}\alpha\right)\left(v_{0},...,v_{k}\right) & = & \sum_{i=0}^{k}\left(-1\right)^{i}v_{i}\left[\alpha\left(v_{0},...,\hat{v}_{i},...,v_{k}\right)\right]\\
 &  & +\sum_{0\leq i<j\leq k}\left(-1\right)^{i+j}\alpha\left(\left[v_{i},v_{j}\right]v_{0},...,\hat{v}_{i},...,\hat{v}_{j},...,v_{k}\right)\end{eqnarray*}
for any vertical vector fields $v_{0},...,v_{k}\in\frak{X}_{p}\left(X\right)$.
This is almost the definition for usual $k$-forms, except that we
restrict $p$-vertical tangent vectors, i.e., sections of $V_{p}$.
This definition makes sense because the Lie bracket $\left[v_{i},v_{j}\right]$
of two such vector fields is itself a section of $V_{p}$. 

The operators $d_{p}$ give rise to a complex \[
C^{\infty}\left(X\right)=\Omega_{p}^{0}\left(X\right)\overset{d_{p}}{\longrightarrow}\Omega_{p}^{1}\left(X\right)...\overset{d_{p}}{\longrightarrow}\Omega_{p}^{n}\left(X\right)\longrightarrow0\]
 and allow to define the spaces of \textbf{fibrewise closed forms
}$\mathcal{Z}_{p}^{k}\left(X\right)=\ker\left(\left.d_{p}\right|_{\Omega_{p}^{k}\left(X\right)}\right)$
and \textbf{fibrewise exact forms }$\mathcal{B}_{p}^{k}\left(X\right)=d_{p}\left(\Omega_{p}^{k-1}\left(X\right)\right)$.

It follows from the definition of $d_{p}$ that, given a fibred form
$\alpha\in\Omega_{p}^{k}\left(X\right)$, the restriction of $d_{p}\alpha$
to any fibre $X_{b}$ is nothing but the usual exterior derivative
of the restriction $\alpha_{b}$, and therefore $d_{p}\alpha$ can
be viewed as the family $\left(d\alpha_{b}\right)_{b\in B}$. In particular,
a fibred form $\alpha\in\Omega_{p}^{k}\left(X\right)$ is closed if
and only if the restrictions $\alpha_{b}\in\Omega^{k}\left(X_{b}\right)$
are closed for all $b\in B$. This justifies the name of the space
$\mathcal{Z}_{p}^{k}\left(X\right)$. 

On the other hand, a exact fibred form $\alpha=d_{p}\beta$ gives
rise to a smooth family of exact forms $\alpha_{b}=d\beta_{b}$. But
the converse is less obvious. Namely, if $\alpha\in\Omega_{p}^{k}\left(X\right)$
is such that all restrictions $\alpha_{b}\in\Omega^{k}\left(X_{b}\right)$
are exact, it is not clear that there exists a smooth family of primitives
$\beta_{b}\in\Omega^{k-1}\left(X_{b}\right)$. This is actually true,
but one needs for this an adapted version of Hodge theory for fibrations,
as explained in the next section.

\subsection{Fibred Hodge decomposition%
\footnote{The results of this section may be old folklore, but we could not
find a precise reference and decided to reproduce them here.%
}}

On a given compact manifold $M$, the usual Hodge theory provides
a way to construct a smooth family of primitives $\beta_{t}\in\Omega^{k-1}\left(M\right)$
for a given smooth family of forms $\alpha_{t}\in\Omega^{k}\left(M\right)$
which are known a priori to be exact for each $t$. It shows at the
same time that the subspace $Z^{1}\left(M\right)$ of closed $k$-forms
is (topologically) complemented%
\footnote{We remind the reader that beyond Hilbert spaces, a closed linear subspace
is not automatically complemented.%
} in $\Omega^{k}\left(M\right)$, i.e., there exists another subspace
$F\subset\Omega^{k}\left(M\right)$ such that $\Omega^{k}\left(M\right)\cong Z^{1}\left(M\right)\oplus F$.
In the context of fibred forms, we have similar results.

For that purpose, it will be convenient to introduce the following
spaces. First of all, the space $W_{p}^{k}\left(X\right)$ defined
in Section \ref{sub-equivalence-with-families} is, as explained above,
a smooth Fréchet bundle over $B$, whose fibres are $\Omega^{k}\left(X_{b}\right)$,
$b\in B$, and we have the natural identification $\Omega_{p}^{k}\left(X\right)=\Gamma\left(B,W_{p}^{k}\right)$.
Collecting the differentials $\Omega^{k}\left(X_{b}\right)$ for all
$b$, one constructs a smooth map, denoted by the same symbol, $d:W_{p}^{k}\left(X\right)\rightarrow W_{p}^{k+1}\left(X\right)$,
which lifts the identity on $B$. On the other hand, we see easily
from the local trivialisations $\left.W_{p}^{k}\left(X\right)\right|_{O}\rightarrow O\times\Omega^{k}\left(X_{b_{0}}\right)$,
where $O\subset B$ and $b_{0}\in O$, explained in Section \ref{sub-equivalence-with-families},
that this map simply corresponds in the trivialisation to the usual
differential on $\Omega^{k}\left(X_{b_{0}}\right)$. This implies
that the following vector bundles \[
Z_{p}^{k}\left(X\right)=\bigcup_{b\in B}Z^{k}\left(X_{b}\right)\,\,\mbox{and}\,\, B_{p}^{k}\left(X\right)=\bigcup_{b\in B}B^{k}\left(X_{b}\right)\]
 over $B$, whose fibres over $b$ are respectively the spaces of
closed forms and of exact forms on $X_{b}$, are actually smooth vector
subbundles of $W_{p}^{k}$. 

Moreover, the differential $d_{p}$ acting on sections of $W_{p}^{k}\left(X\right)$
is just the left composition by $d:W_{p}^{k}\left(X\right)\rightarrow W_{p}^{k+1}\left(X\right)$.
This shows in particular the identification \[
\mathcal{Z}_{p}^{k}\left(X\right)=\Gamma\left(B,Z_{p}^{k}\left(X\right)\right).\]
To proceed further, we need a kind of Hodge theory adapted to our
context of fibrations. 
\begin{lem}
\label{lem-hodge-inverse-d}Assume $B$ is compact. For each $k$,
there exists a smooth vector bundle morphism $\delta:W_{p}^{k+1}\left(X\right)\rightarrow W_{p}^{k}\left(X\right)$,
such that \[
d\circ\delta\circ d=d.\]
In other words, $\delta$ inverts $d:W_{p}^{k}\left(X\right)\rightarrow B_{p}^{k+1}\left(X\right)$
from the right. \end{lem}
\begin{proof}
First of all, we cover the base space $B$ by open subsets $B_{i}\subset B$,
$i\in I$, such that $W_{p}^{k}\left(X\right)$ is trivial over each
$B_{i}$, with trivialisations of the form $\Phi_{i}:\left.W_{p}^{k}\left(X\right)\right|_{B_{i}}\rightarrow B_{i}\times\Omega^{k}\left(X_{b_{i}}\right)$,
where $b_{i}\in O$, as explained in Section \ref{sub-equivalence-with-families}.
We can assume that the cover is finite, since $B$ is compact. For
convenience, we also denote by $\tilde{\Phi}_{i}$ the second component
of $\Phi_{i}$.

Now, for each $i\in I$ we know from the standard Hodge theory, that
there exists a map%
\footnote{For example take $\delta_{i}=Gd^{*}$ where $G$ is the green operator
of the Laplacian and $d^{*}$ is the formal $L^{2}$-adjoint of $d$.%
} $\delta_{i}:\Omega^{k+1}\left(X_{b_{i}}\right)\rightarrow\Omega^{k}\left(X_{b_{i}}\right)$
which satisfies $d\delta_{i}d=d$. Then, let $\rho_{i}$ be a partition
of unity subordinate to the $B_{i}$'s, and let us define $\delta$
as follows. For any $\alpha\in W_{p}^{k}\left(X\right)$ above a point
$b\in B$ set \[
\delta\left(\alpha\right)=\sum_{i\in I}\rho_{i}\left(b\right)\Phi_{i}^{-1}\left(b,\delta_{i}\left(\tilde{\Phi}_{i}\left(\alpha\right)\right)\right).\]
 This is a smooth map since the $\rho_{i}$ have compact support in
$B_{i}$, precisely where the $\Phi_{i}$'s are defined. Moreover,
since the maps $\Phi_{i}$ preserve the differential, one has for
any $\alpha\in W_{p}^{k}\left(X\right)$ over $b$ \begin{eqnarray*}
d\circ\delta\circ d\left(\alpha\right) & = & d\left(\sum_{i\in I}\rho_{i}\left(b\right)\Phi_{i}^{-1}\left(b,\delta_{i}\left(\tilde{\Phi}_{i}\left(d\alpha\right)\right)\right)\right)\\
 & = & \sum_{i\in I}\rho_{i}\left(b\right)\Phi_{i}^{-1}\left(b,d\delta_{i}d\left(\tilde{\Phi}_{i}\left(\alpha\right)\right)\right).\end{eqnarray*}
Then, using $d\delta_{i}d=d$ on each $\Omega^{k+1}\left(X_{b_{i}}\right)$
and the partition of unity, one gets easily $d\circ\delta\circ d=d$
on $W_{p}^{k}\left(X\right)$.
\end{proof}
This result transposes easily to the space of sections of $W_{p}^{k}\left(X\right)$.
Indeed, if we denote by $\delta_{p}:\Omega_{p}^{k+1}\left(X\right)\rightarrow\Omega_{p}^{k}\left(X\right)$
the left composition by $\delta$, we obtain immediately \[
d_{p}\delta_{p}d_{p}=d_{p}.\]
It follows from this fact that, given a fibred form $\alpha\in\Omega_{p}^{k}\left(X\right)$
such the restriction $\alpha_{b}$ is exact for all $b$, then $\delta_{p}\alpha\in\Omega_{p}^{k+1}\left(X\right)$
is a smooth family of primitives. 

On the other hand, the map $P=\delta_{p}d_{p}$ is a continuous linear
map from $\Omega_{p}^{k}\left(X\right)$ to itself. It is obviously
a projection $P^{2}=P$ and moreover its kernel is exactly $\mathcal{Z}_{p}^{k}\left(X\right)$.
Indeed, if $P\alpha=0$ this implies that $d_{p}\alpha$ vanishes
since $\delta_{p}$ is injective when restricted on $\mathcal{B}_{p}^{k}\left(X\right)$.
This shows $\ker P\subset\mathcal{Z}_{p}^{k}\left(X\right)$ and the
other inclusion is trivial. Therefore, the map $1-P$ is a continuous
projection of $\Omega_{p}^{k}\left(X\right)$ onto $\mathcal{Z}_{p}^{k}\left(X\right)$.
This is equivalent to the following.
\begin{thm}
\label{lem-fibred-hodge-decomposition}Assume $B$ is compact. The
set of fibrewise closed forms $\mathcal{Z}_{p}^{k}\left(X\right)$
is a closed and complemented subspace of $\Omega_{p}^{k}\left(X\right)$. 
\end{thm}

\subsection{fibred diffeomorphisms }
\begin{defn}
Let $p:X\rightarrow B$ and $q:Y\rightarrow B$ two fibrations over
the same base manifold. A smooth map $\varphi:X\rightarrow Y$ is
called \textbf{vertical} whenever $q\circ\varphi=p$, i.e., when $\varphi$
is a bundle morphism.
\end{defn}
Such a $\varphi$ induces a smooth family of smooth maps $\varphi_{b}:X_{b}\rightarrow Y_{b}$.
It is then easy to check that it makes sense to define the pullback
operator $\varphi^{*}:\Omega_{q}^{k}\left(Y\right)\rightarrow\Omega_{p}^{k}\left(X\right)$
through $\left(\varphi^{*}\alpha\right)_{b}=\varphi_{b}^{*}\alpha_{b}$.
In particular, when $X=Y$, we can associate to any one-parameter
family of vertical diffeomorphisms $\phi^{t}$ a (time-dependent)
vertical vector field $Z_{t}$ and one gets easily convinced that
for any fibred form $\alpha\in\Omega_{p}^{k}\left(X\right)$, one
has \[
\frac{d}{dt}\left(\phi_{Z_{t}}^{t}\right)^{*}\alpha=\left(\phi_{Z_{t}}^{t}\right)^{*}\mathcal{L}_{Z_{t}}\alpha\]
where $\mathcal{L}_{Z_{t}}\alpha$ is the fibrewise Lie derivative,
defined by $\left(\mathcal{L}_{Z_{t}}\alpha\right)_{b}=\mathcal{L}_{Z_{t}}\alpha_{b}$,
with $Z_{t}$ viewed as a vector field on the fibre $X_{b}$. In particular
it satisfies the fibred Cartan formula : $\mathcal{L}_{Z}\alpha=Z\lrcorner d_{p}\alpha+d_{p}\left(Z\lrcorner\alpha\right)$
for any $Z\in\frak{X}_{p}\left(X\right)$ and $\alpha\in\Omega_{p}^{k}\left(X\right)$.

\section{\label{sec-WTN-symplectic-bundles}WTN in symplectic bundles}

\subsection{\label{sub-symplectic-bundles}Symplectic bundles and Lagrangian
subbundles}
\begin{defn}
\label{def-symplectic-bundle}A \textbf{symplectic bundle} is a (locally
trivial) fibration $p:X\rightarrow B$ equipped with a fibred $2$-form
$\omega\in\Omega_{p}^{2}\left(X\right)$, such that the restrictions
$\omega_{b}\in\Omega^{2}\left(X_{b}\right)$ are symplectic for all
$b\in B$, i.e., closed and non-degenerate. 
\end{defn}
In particular $d_{p}\omega=0$, i.e., $\omega\in\mathcal{Z}_{p}^{2}\left(X\right)$.
Notice that the restrictions $\omega_{b}$ may not be diffeomorphic
to each other.
\begin{defn}
\label{def-lag-subbundle}A subbundle $L\subset X$ is called \textbf{Lagrangian
}when the restrictions $L_{b}=X_{b}\cap L$ are Lagrangian in $X_{b}$
for all $b\in B$. 
\end{defn}
The following canonical example is precisely the local model for neighbourhoods
of Lagrangian subbundles, as shown in Section \ref{sub-WTN-lag-subbundle}.
\begin{example}
\label{exa-fibred-cotangent-bundle}Let $p:M\rightarrow B$ be a fibration.
The vector bundle $X=V_{p}^{*}\left(M\right)\overset{\pi}{\rightarrow}M$
dual to the $p$-vertical tangent bundle of $M$ is in fact a symplectic
bundle over $B$, with projection $p\circ\pi$. Indeed, for any $b\in B$
the fibre $\left(p\circ\pi\right)^{-1}\left(b\right)$ is nothing
but the cotangent bundle $T^{*}M_{b}$ of $M_{b}=p^{-1}\left(b\right)$,
hence in particular a symplectic manifold.\end{example}
\begin{defn}
\label{def:fibred-Liouville-form}A fibred Liouville form on $X$
is a $1$-form $\lambda\in\Omega_{p}^{1}\left(X\right)$ such that
$d_{p}\lambda=\omega$. It corresponds to a family of primitives $\lambda_{b}$
of the symplectic forms $\omega_{b}$.
\end{defn}

\subsection{A fibred Poincaré lemma}

We will make use of the following technical lemma, whose proof can
be found in \cite{humiliere_roy_2}.
\begin{lem}
\label{lem-voisinage-pratique}Let $p:X\rightarrow B$ be a fibre
bundle and $L\subset X$ a subbundle. Then, there exists a tubular
neighbourhood $U$ of $L$, whose associated projection $q:U\rightarrow L$
is $p$-vertical.

\noindent %
\begin{minipage}[t]{1\textwidth}%
\begin{center}
\includegraphics[scale=0.6]{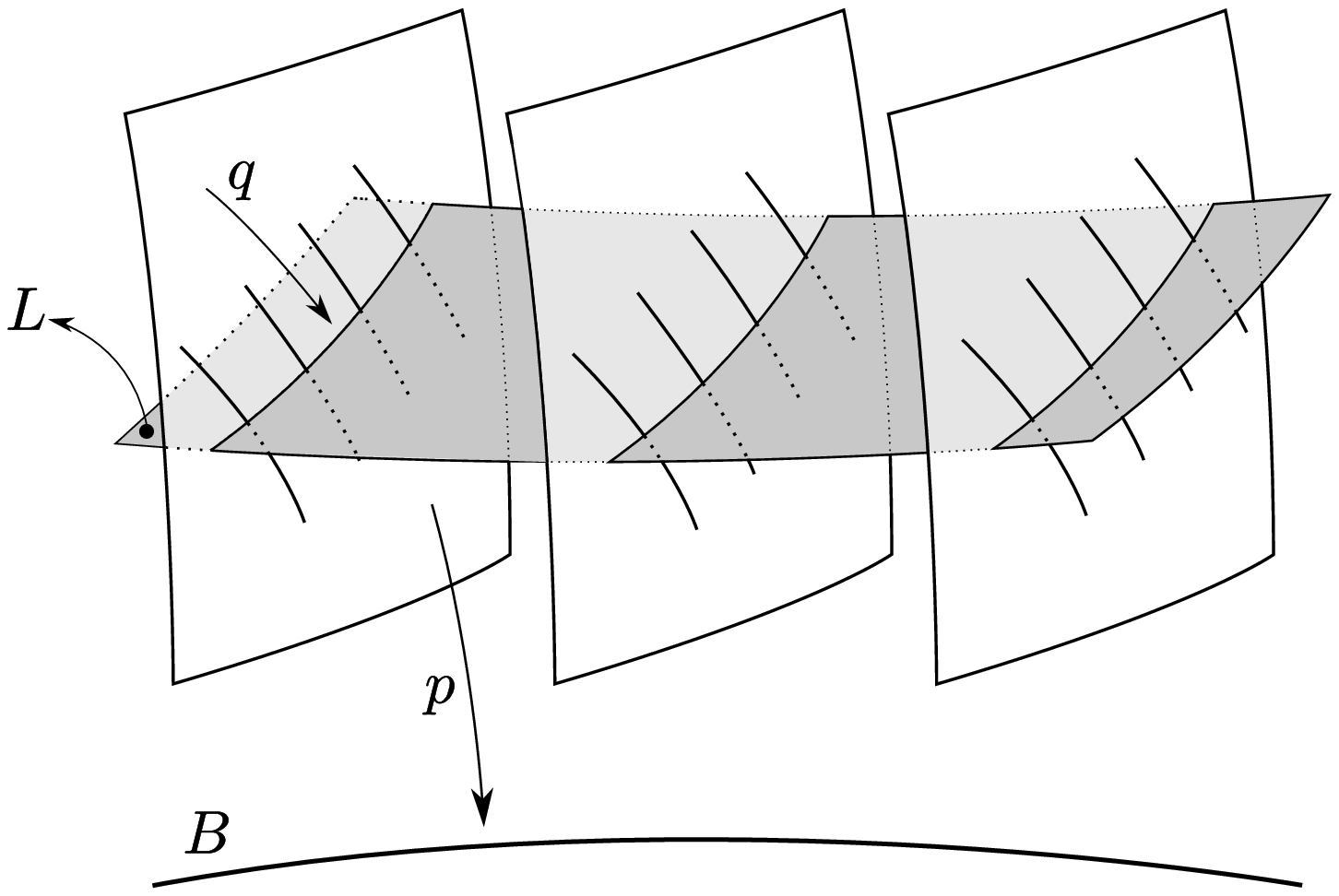}
\par\end{center}%
\end{minipage}
\end{lem}
We now give a Poincaré Lemma for fibred forms.
\begin{lem}
\label{lem-existe-liouville-form}Let $L\subset X$ a Lagrangian subbundle
of a symplectic bundle $p:X\rightarrow B$. There exist a tubular
neighbourhood $U$ of $L$ and a fibred Liouville form $\lambda\in\Omega_{p}^{1}\left(U\right)$
vanishing on $L$. \end{lem}
\begin{proof}
First of all, one can find a deformation retraction of a neighbourhood
$U\subset X$ of $L$ onto $L$, i.e. a smooth map $\rho:\left[0,1\right]\times U\rightarrow U$,
satisfying $\rho_{0}=\mathbb{I}$, $\rho_{1}\left(U\right)=L$ and
$\left.\rho_{t}\right|_{L}=\mathbb{I}$ for all $t$, where $\rho_{t}:U\rightarrow U$
is defined by $\rho_{t}=\rho\left(t,.\right)$. Moreover, one can
assume that the deformation preserves the fibres of $p$, i.e. $p\circ\rho_{t}=p$
for all $t$. Indeed, take a tubular neighbourhood $U$ whose projection
$\rho$ si $p$-vertical, as in in Lemma \ref{lem-voisinage-pratique},
and use the dilation with a factor depending on $t$, in the vector
bundle associated to this tubular neighbourhood. 

We adapt now to the category of fibred forms a standard argument in
the spirit of Poincaré Lemma, by writing for any $\alpha\in\Omega_{p}^{k}\left(X\right)$
\[
\rho_{1}^{*}\alpha-\rho_{0}^{*}\alpha=\int_{0}^{1}\left(\frac{d}{dt}\rho_{t}^{*}\alpha\right)\, dt.\]
Then, we make use of the decomposition $\rho_{t}=\rho\circ T_{t}\circ\iota$
where $\iota:U\rightarrow\mathbb{R}\times U$ is the injection $\iota\left(x\right)=\left(0,x\right)$
and $T_{t}:\mathbb{R}\times U\rightarrow\mathbb{R}\times U$ is the
translation $T_{t}\left(s,x\right)=\left(s+t,x\right)$. Viewing $\mathbb{R}\times U$
as a fibration over $B$ in the obvious way $\left(t,x\right)\mapsto p\left(x\right)$,
we see that the three smooth maps $\rho$, $T_{t}$ and $\iota$ are
vertical. Therefore, we have $\frac{d}{dt}\rho_{t}^{*}\alpha=\iota^{*}T_{t}^{*}\mathcal{L}_{\partial_{t}}\rho^{*}\alpha$,
where $\partial_{t}$ denotes the (time-independent) vector field
of the one-parameter group of diffeomorphisms $T_{t}$. Then, applying
the fibred Cartan formula, we finally obtain\[
\rho_{1}^{*}\alpha-\rho_{0}^{*}\alpha=Pd_{p}\alpha+d_{p}P\alpha,\]
 where the homotopy operator $P:\Omega_{p}^{\bullet}\left(X\right)\rightarrow\Omega_{p}^{\bullet-1}\left(X\right)$
is defined by \[
P\beta=\int_{0}^{1}\iota^{*}T_{t}^{*}\partial_{t}\lrcorner\rho^{*}\beta\, dt.\]
 Moreover, by construction $P\beta$ vanishes on $L$. Indeed, to
compute $P\beta$ at a point $x\in L$, we need to evaluate $\partial_{t}\lrcorner\rho^{*}\beta$
at $\left(t,x\right)$ for all $t\in\left[0,1\right]$, or equivalently
$D_{\left(t,x\right)}\rho\left(\partial_{t}\right)\lrcorner\beta\left(\rho\left(t,x\right)\right)$.
But the fact that $\rho\left(t,x\right)=x$ on $L$ implies that $D_{\left(t,x\right)}\rho\left(\partial_{t}\right)$
vanishes for each $t$, therefore $P\beta$ vanishes on $L$. 

Finally, we apply this Poincaré Lemma to the fibred symplectic form
$\omega\in\Omega_{p}^{2}\left(X\right)$. Since $L$ is Lagrangian,
one has $\rho_{1}^{*}\omega=0$ and the closedness of $\omega$ provides
\[
\omega=-d_{p}P\omega,\]
which proves the lemma with $\lambda=-P\omega$.
\end{proof}

\subsection{Transverse polarisations}

In the sequel a {}``foliation $P$'' denotes actually the (integrable)
distribution of tangent planes $P_{x}\subset T_{x}X$ to the leaves. 
\begin{defn}
Let $p:X\rightarrow B$ be a symplectic bundle. A \textbf{polarisation}
$P$ is a foliation of $X$ whose leaves are all included in the fibres
of $p$ and Lagrangian.
\end{defn}
In other words, it is a smooth family (parametrized by $B$) of Lagrangian
foliations. On the other hand, it is known that the existence of a
Lagrangian foliation transverse to a given Lagrangian submanifold
is intimately related to the existence of a Liouville form vanishing
on the submanifold. Here is a fibred version of this fact. 
\begin{lem}
\label{lem-exist-transverse-pola}Let $p:X\rightarrow B$ be a symplectic
bundle and $L\subset X$ a Lagrangian subbundle. If $\lambda\in\Omega_{p}^{1}\left(X\right)$
is a fibred Liouville form on $X$ vanishing on the manifold $L$,
then there exists near $L$ a unique polarisation which is transverse
to $L$ and everywhere included in $\ker\lambda$.\end{lem}
\begin{proof}
Denote by $\omega\in\Omega_{p}^{2}\left(X\right)$ the fibred symplectic
form. For any $b\in B$, the Liouville form $\lambda_{b}\in\Omega^{1}\left(X_{b}\right)$
corresponds through the isomorphism $\omega_{b}:TX_{b}\rightarrow T^{*}X_{b}$
to a vector field $Y_{b}\in\mathfrak{X}\left(X_{b}\right)$ vanishing
on $L_{b}$. We adapt an argument from \cite[p. 230]{guillemin_sternberg_2}.
At each point $x\in L_{b}$, the Jacobian $J_{b,x}=\left[Y_{b},.\right]_{x}$
actually defines a linear map $J_{b,x}:T_{x}X_{b}\rightarrow T_{x}X_{b}$.
Following \cite[p. 230]{guillemin_sternberg_2}, one can show that
$T_{x}X_{b}$ decomposes into two eigenspaces, namely $T_{x}L_{b}$
with eigenvalue $0$ and a transverse Lagrangian plane $P_{x}$ with
eigenvalue $1$. Now, we can obviously view the family $Y_{b}$ simply
as a vector field $Y$ on $X$, vanishing on $L$ and everywhere tangent
to the fibres of $p$. It is then easy to see that at any point $x\in L$,
the Jacobian $J_{x}=\left[Y,.\right]_{x}$ provides an eigenspace
decomposition of $T_{x}X$ into $T_{x}L$ with eigenvalue $0$ and
$P_{x}$ eigenvalue $1$. Lastly, we can use a result of \cite{sternberg_1}
on normal forms of singular vector fields, which implies in particular
that there exists a unique transverse foliation to $L$, in a neighbourhood
of $L$, which is invariant by $Y$.

It remains now to show that the foliation is tangent to the fibres
of $p$ and Lagrangian in each fibre. First of all, if the foliation
was not tangent to the fibres of $p$, then, because it is $Y$-invariant,
it would contradict the fact that it is transverse to $L$. Indeed,
if $Z$ is a vector in the foliation but not tangent to the fibres
of $p$, then the vector $D\phi_{Y}^{-t}\left(Z\right)$ would tend,
for $t\rightarrow\infty$, to a vector tangent to $L$ but also to
$P$. This is because, as $t\rightarrow\infty$ the flows $\phi_{Y}^{-t}$
contracts a neighbourhood of $L$ onto $L$. On the other hand, $L$
and the leaves of $P$ have complementary dimension in $X$, since
$L_{b}$ and the leaves of $P$ have complementary dimension in $X_{b}$.
This contradiction implies therefore that the leaves of $P$ are contained
in the fibres of $p$.

Second, this foliation must be Lagrangian in each $X_{b}$. Indeed,
for any two vectors $Z_{1},Z_{2}\in T_{x}X_{b}$ at some point $x\in X_{b}$,
one has \[
\omega_{b}\left(\left(\phi_{Y_{b}}^{-t}\right)_{*}Z_{1},\left(\phi_{Y_{b}}^{-t}\right)_{*}Z_{2}\right)=O\left(e^{-2t}\right)\]
 on the one hand and \[
\left(\phi_{Y_{b}}^{-t}\right)^{*}\omega_{b}\left(Z_{1},Z_{2}\right)=e^{-t}\omega_{b}\left(Z_{1},Z_{2}\right)\]
 on the other hand, since $Y_{b}$ is a Liouville vector field, and
thus symplectic conformal. This implies that $\omega_{b}\left(Z_{1},Z_{2}\right)=0$.

To conclude, notice that the Liouville vector field $X_{b}$ is tangent
to the polarisation if and only if the latter is included in $\ker\lambda_{b}$.
\end{proof}

\subsection{\label{sub-WTN-lag-subbundle}WTN for Lagrangian subbundles}
\begin{thm}
\label{thm-WTN-for-lag-subbundle}Let $p:X\rightarrow B$ be a symplectic
bundle. Let $L\subset X$ be a Lagrangian subbundle and $p_{L}:L\rightarrow B$
the projection. Let $V_{p_{L}}^{*}\left(L\right)\overset{\pi}{\rightarrow}L$
be the $p_{L}$-vertical cotangent bundle of $L$. Then, there exist
a neighbourhood $U\subset X$ of $L$, a neighbourhood $V\subset V_{p_{L}}^{*}\left(L\right)$
of the $0$-section of $\pi$ and a diffeomorphism $\varphi:V\rightarrow U$
with the following properties : 
\begin{itemize}
\item The image of the $0$-section is $L$.
\item The following diagram commutes \begin{eqnarray*}
\varphi:V & \rightarrow & U\\
\downarrow p_{L}\circ\pi &  & \downarrow p\\
B & \overset{\mathbb{I}}{\rightarrow} & B\end{eqnarray*}

\item For each $b\in B$, the restriction \[
\varphi_{b}:\left(p_{L}\circ\pi\right)^{-1}\left(b\right)\cap V\rightarrow p^{-1}\left(b\right)\cap U\]
 is symplectic.
\end{itemize}
Therefore, since $\left(p_{L}\circ\pi\right)^{-1}\left(b\right)=T^{*}L_{b}$,
the restrictions $\varphi_{b}$ form indeed a smooth family of WTN
around the Lagrangian submanifolds $L_{b}=p_{L}^{-1}\left(b\right)$.

\noindent %
\begin{minipage}[t]{1\textwidth}%
\begin{center}
\includegraphics[scale=0.7]{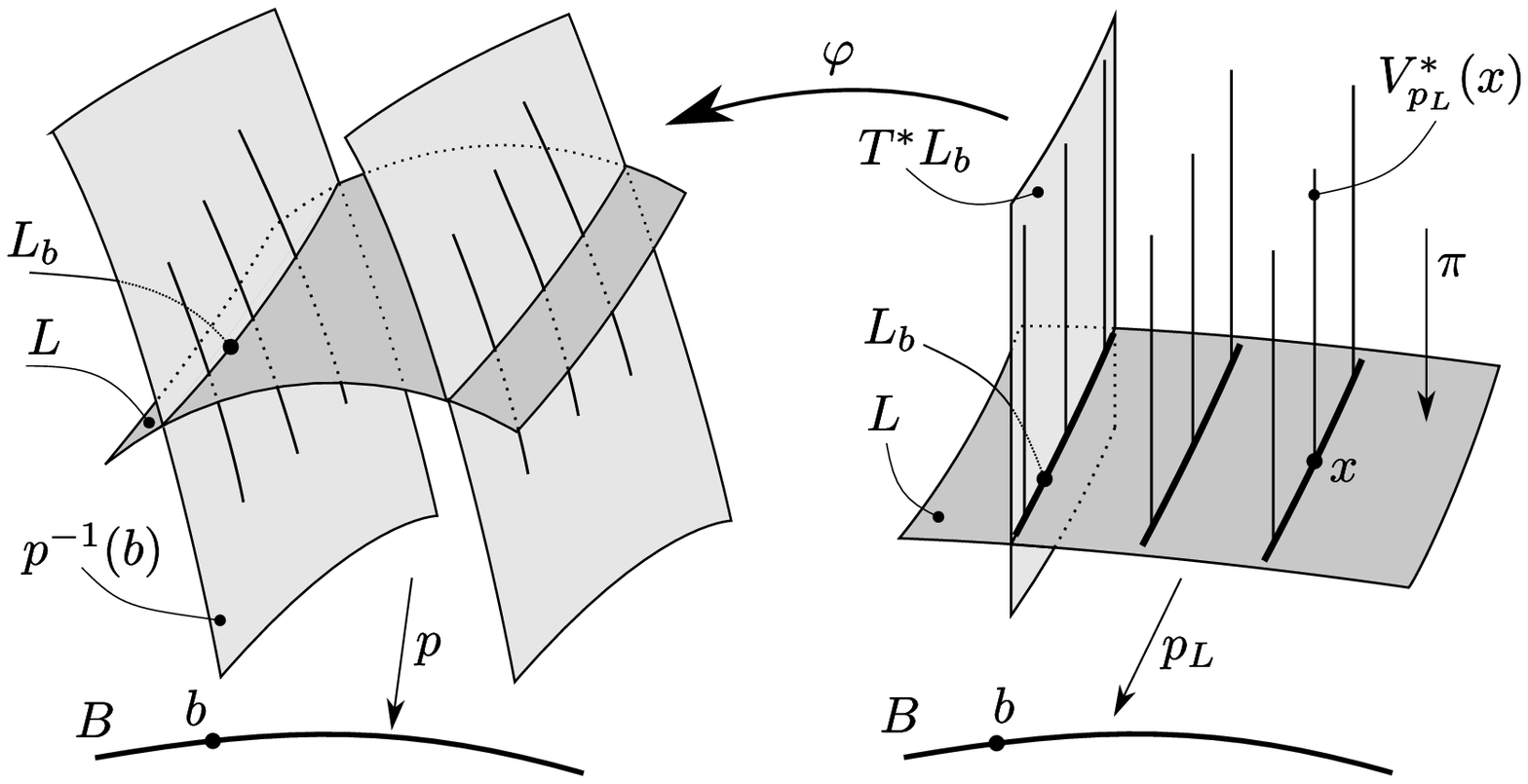}
\par\end{center}%
\end{minipage}\end{thm}
\begin{proof}
To construct this family of WTN around the fibres $L_{b}$, we will
adapt a construction from \cite[sec. 4.4]{woodhouse1}. First of all,
Lemma \ref{lem-existe-liouville-form} tells us that there exists
a Liouville form $\lambda_{b}\in\Omega_{p}^{1}\left(U\right)$ in
a neighbourhood $U$ of $L$ and vanishing on $L$. Thanks to Lemma
\ref{lem-exist-transverse-pola}, this implies the existence in $U$
of a transverse polarisation $P$. Now, for any $b\in B$, the restriction
to $X_{b}\cap U$ of $P$ is a polarisation. We know from \cite{weinstein1}
that any leaf is naturally endowed with an affine structure. This
allows to do parallel transport and thus define exponential maps.
At any point $x\in L_{b}$, we have two transversal Lagrangian planes
in $T_{x}X_{b}$, namely the tangent spaces to $L_{b}$ and to $P$
respectively. Using the natural isomorphism between $T_{x}^{*}L_{b}$
and $P_{x}$, we construct $\varphi$ as follows. For any $\alpha\in V_{p_{L}}^{*}\left(x\right)$,
which is canonically an element of $T_{x}^{*}L_{b}$, we associate
its corresponding tangent vector $Y\in P_{x}$. Then, we take the
exponential map $\exp_{x}Y$ and obtain a point in $X_{b}$, on the
same leaf of $P$ as $x$. Of course, $Y$ has to be sufficiently
small, so that the exponential map still makes sense. This means that
our construction is well-defined provided $\alpha$ is taken in a
sufficiently small neighbourhood $V$ of the $0$-section. Up to taking
smaller $V$ and $U$, one can insure that $\varphi$ is a diffeomorphism
from $V$ to $U$. Moreover, by construction $\varphi$ sends $\left(p_{L}\circ\pi\right)^{-1}\left(b\right)\cap V$
to $p^{-1}\left(b\right)\cap U$. Then, following \cite[sec. 4.4]{woodhouse1},
one can show that for each $b\in B$, the restriction $\varphi_{b}$
is symplectic. 
\end{proof}

\subsection{\label{sub-manifold-structure-Subb}The Fréchet manifold of Lagrangian
subbundles}

Given a fibration $p:X\rightarrow B$ we will denote by $\mbox{Subb}\left(X\right)$
the space of all compact subbundles $L$ of $X$. Because of Ehresmann
Theorem \cite{ehresmann} this is exactly the space of compact submanifolds
$L\subset X$ such that the restriction of the projection $\left.p\right|_{L}:L\rightarrow B$
is a submersion. The notation $\mbox{Subb}\left(X\right)$ is slightly
imprecise, since $X$ denotes here the whole fibration structure,
i.e., the manifolds $X$ and $B$, and the projection $p$. 

The space $\mbox{Subb}\left(X\right)$ is a subset of the space $\mathcal{S}\left(X\right)$
of all compact submanifolds of $X$, which is known to be a Fréchet
manifold, see \cite[p. 92]{hamilton} or \cite{michor3}. We prove
now that it is actually an open set and therefore itself a Fréchet
manifold. Indeed, the manifold structure on $\mathcal{S}\left(X\right)$
near some $L_{0}\subset X$ is given through an identification of
the $C^{1}$-close submanifolds $L\subset X$ with sections of a tubular
neighbourhood $N$ of $L_{0}$ in $X$. For any $L$ $C^{1}$-close
to $L_{0}$, denote by $\alpha\in\Gamma\left(L_{0},N\right)$ the
corresponding section. Now, $\left.p\right|_{L}:L\rightarrow B$ is
a submersion if and only if $p\circ\alpha:L_{0}\rightarrow B$ is
so, because $\alpha$ defines a diffeomorphism from $L_{0}$ to $L$.
On the other hand, the condition of $p\circ\alpha$ being a submersion
is equivalent to the $1$-jet $j^{1}\left(p\circ\alpha\right)$ having
an image included in an open set in $J^{1}\left(L_{0},B\right)$ (the
open set composed of linear maps $TL_{0}\rightarrow TB$ with maximal
rank). Finally, the map $\alpha\mapsto j^{1}\left(p\circ\alpha\right)$
is clearly continuous, and it follows that $\mbox{Subb}\left(X\right)$
is an open subset of $\mathcal{S}\left(X\right)$. 

Now, suppose $X$ is a symplectic bundle. Thanks to the WTN of Theorem
\ref{thm-WTN-for-lag-subbundle}, we can investigate the differentiable
structure of the space $\mbox{Subb}_{Lag}\left(X\right)$ of Lagrangian
subbundles.
\begin{thm}
\label{thm-lag-subbundle-sous-varioche}$\mbox{Subb}_{Lag}\left(X\right)$
is a Fréchet submanifold of $\mbox{Subb}\left(X\right)$.\end{thm}
\begin{proof}
Near any Lagrangian subbundle $L_{0}\subset X$, we will construct
a local chart $\hat{\varphi}$ of $\mbox{Subb}\left(X\right)$ as
follows. Consider the local diffeomorphism $\varphi:V_{p_{L_{0}}}^{*}\left(L_{0}\right)\supset V\rightarrow U\subset X$
of Theorem \ref{thm-WTN-for-lag-subbundle} which sends the graph
of the $0$-section to $L_{0}$. Then, the map $\hat{\varphi}$ which
sends a section $\alpha\in\Gamma\left(L_{0},V_{p_{L_{0}}}^{*}\right)$
(with image in $V$) to the submanifold $\varphi\left(\alpha\left(L_{0}\right)\right)\subset U$
is indeed a chart compatible with the manifold structure of $\mbox{Subb}\left(X\right)$,
since the diffeomorphism $\varphi$ gives $U$ the structure of (an
open set in) vector bundle over $L_{0}$, and any $C^{1}$-close submanifold
$L$ corresponds through $\hat{\varphi}^{-1}$ to a section of this
vector bundle.

Now, a subbundle $L$ is Lagrangian if and only if each $L_{b}\subset X_{b}$
is Lagrangian. Let $\alpha=\hat{\varphi}^{-1}\left(L\right)$ be the
corresponding section of $V_{p_{L_{0}}}^{*}$. As mentioned earlier,
for each $b$ the restriction $\alpha_{b}$ to $\left(L_{0}\right)_{b}=\left.p\right|_{L_{0}}^{-1}\left(b\right)$
is simply a $1$-form on $\left(L_{0}\right)_{b}$. On the other hand,
the graph of $\alpha_{b}$ is precisely the image of $L_{b}$ through
the restriction $\varphi_{b}$. But since each $\varphi_{b}$ is symplectic,
$L_{b}$ is Lagrangian in $X_{b}$ if and only if the $\mbox{gr}\left(\alpha_{b}\right)$
is Lagrangian in $T^{*}\left(L_{0}\right)_{b}$, i.e., $\alpha_{b}$
is closed. This shows that $L$ is Lagrangian if and only the corresponding
fibred form $\alpha\in\Omega_{p_{L_{0}}}^{1}\left(L_{0}\right)$ is
closed, i.e., $d_{p_{L_{0}}}\alpha=0$. Finally, we know from Theorem
\ref{lem-fibred-hodge-decomposition} that the space of fibrewise
closed $1$-forms $\mathcal{Z}_{p}^{1}$ is complemented in $\Omega_{p_{L_{0}}}^{1}\left(L_{0}\right)$.
In particular it is a Fréchet submanifold of $\Omega_{p_{L_{0}}}^{1}\left(L_{0}\right)$.
This concludes the proof, since $\hat{\varphi}$ is a local diffeomorphism.
\end{proof}

\section{\label{sec-lag-fibrations}Space of Lagrangian fibrations}

\subsection{Lagrangian fibrations and Lagrangian subbundles}

A submersion $\pi:M\rightarrow B$ between two compact manifolds $M$
and $B$ is automatically a locally trivial fibration because of Ehresmann
Theorem \cite{ehresmann}. Denote by $\mbox{Fib}\left(M,B\right)$
the set of fibrations from $M$ to $B$. It is an open set of $C^{\infty}\left(M,B\right)$
and therefore a Fréchet submanifold. 
\begin{defn}
Suppose $M$ is a symplectic manifold. A \textbf{Lagrangian fibration}
is a fibration $\pi:M\rightarrow B$ whose fibres $\pi^{-1}\left(b\right)$,
$b\in B$, are all Lagrangian submanifolds of $M$. Denote by $\mbox{Fib}_{\mbox{Lag}}\left(M,B\right)$
the space of Lagrangian fibrations.
\end{defn}
The aim of this section is to prove Corollary \ref{thm-fib_lag-submanifold-of-fib}
which states that $\mbox{Fib}_{\mbox{Lag}}\left(M,B\right)$ is a
Fréchet submanifold of $\mbox{Fib}\left(M,B\right)$. For this, we
can use the results of the previous sections, because of the following
parametrisation. For any smooth map $\pi\in C^{\infty}\left(M,B\right)$
we associate the map $\hat{\pi}=\left(\mathbb{I},\pi\right)$ which
is a smooth section of $M\times B$ viewed as a (trivial) bundle over
$M$. Notice that the correspondence $\pi\leftrightarrow\hat{\pi}$
is a Fréchet diffeomorphism between $C^{\infty}\left(M,B\right)$
and $\Gamma\left(M,M\times B\right)$. 

Now, $M\times B$ is a (trivial) symplectic bundle with respect to
the second projection $p:M\times B\rightarrow B$. On the other hand,
given any $\pi\in C^{\infty}\left(M,B\right)$, consider the graph
$L_{\pi}\subset M\times B$ of the corresponding $\hat{\pi}\in\Gamma\left(M,M\times B\right)$. 

\begin{center}
\includegraphics[scale=0.7]{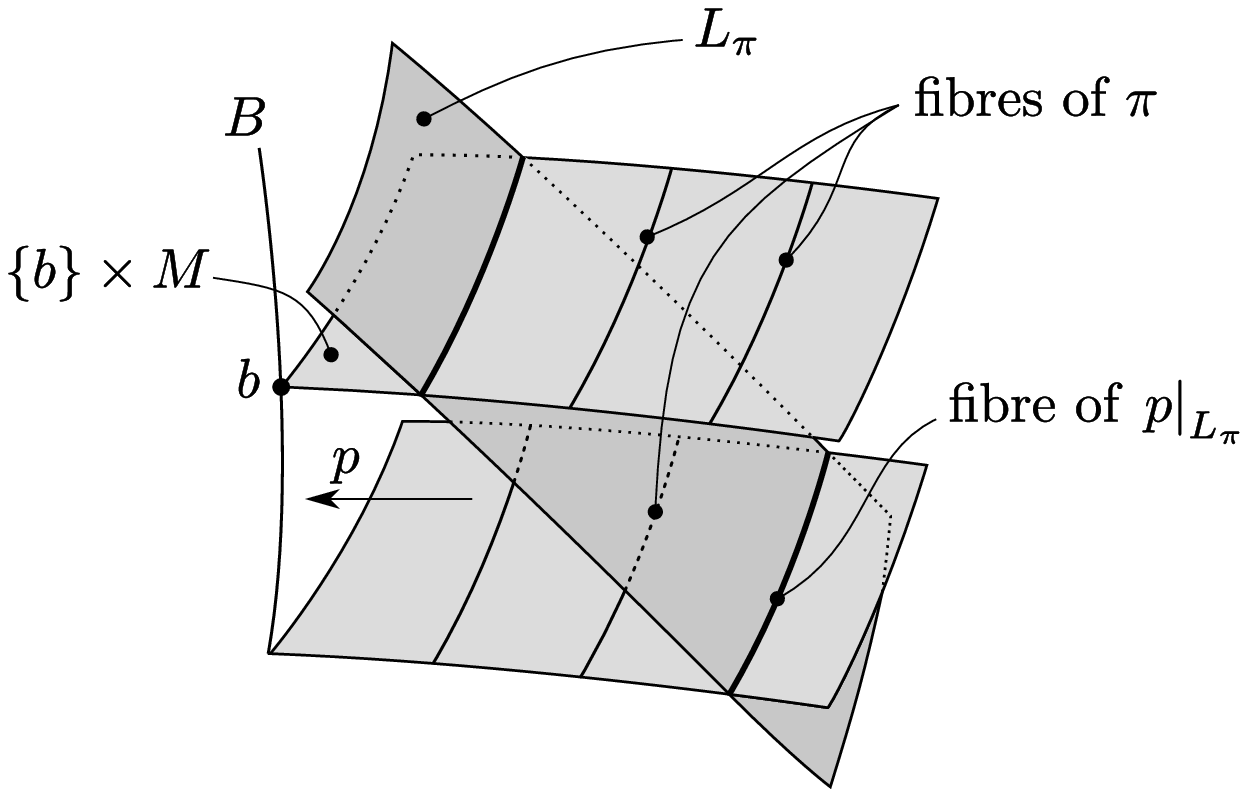}
\par\end{center}

\noindent It is easy to see that $\pi$ is a submersion if and only
if $L_{\pi}$ is a subbundle of $p:M\times B\rightarrow B$. Indeed,
$L_{\pi}$ is a subbundle when the restriction of the projection $\left.p\right|_{L_{\pi}}:L_{\pi}\rightarrow B$
is a submersion. But on the other hand, one has the relation $\pi=p\circ\hat{\pi}$.
Therefore, $\left.p\right|_{L_{\pi}}:L_{\pi}\rightarrow B$ is a submersion
if and only if $\pi$ is so, since $\hat{\pi}$ is an embedding from
$M$ into $M\times B$, whose image is precisely $L_{\pi}$.

Finally, whenever $\left.p\right|_{L_{\pi}}:L_{\pi}\rightarrow B$
is a submersion, its fibre over $b\in B$ is $\left\{ \left(m,b\right)\mid b=\pi\left(m\right)\right\} $,
i.e., $\pi^{-1}\left(b\right)\times\left\{ b\right\} $. Therefore,
the fibres of $\left.p\right|_{L_{\pi}}$ are Lagrangian in the fibres
of $p$ if and only if the fibres of $\pi$ are Lagrangian. We have
thus proved the following :
\begin{lem}
\label{lem-equiv-subbundle-fibration}The submanifold $L_{\pi}\subset M\times B$
is a (Lagrangian) subbundle of $p:M\times B\rightarrow B$ if and
only $\pi$ is a (Lagrangian) fibration.
\end{lem}
To complete the argument, it remains to show that the parametrisation
$\pi\mapsto L_{\pi}$ is a Fréchet diffeomorphism. But this relies
on the following general lemma.
\begin{lem}
\label{lem-bi-fibration}Let $N$ a manifold equipped with two fibre
bundle structures $p_{1}:N\rightarrow L$ and $p_{2}:N\rightarrow L$
over the same compact manifold $L$. Suppose in addition that there
exists a map $\iota_{L}:L\rightarrow N$ which is a section of both
fibrations, i.e. $p_{1}\circ\iota_{L}=p_{2}\circ\iota_{L}=\mathbb{I}_{L}$.
Then, the Fréchet manifolds of sections $\Gamma\left(p_{1}\right)$
and $\Gamma\left(p_{2}\right)$ are locally diffeomorphic around $\iota_{L}$.
An explicit diffeomorphism is \begin{eqnarray*}
\Psi:\Gamma\left(p_{1}\right)\supset V_{1} & \rightarrow & V_{2}\subset\Gamma\left(p_{2}\right)\\
\alpha & \mapsto & \alpha\circ\left(p_{2}\circ\alpha\right)^{-1}\end{eqnarray*}
and its inverse is $\Psi^{-1}:\beta\mapsto\beta\circ\left(p_{1}\circ\beta\right)^{-1}$.

\noindent %
\begin{minipage}[t]{1\textwidth}%
\begin{center}
\includegraphics{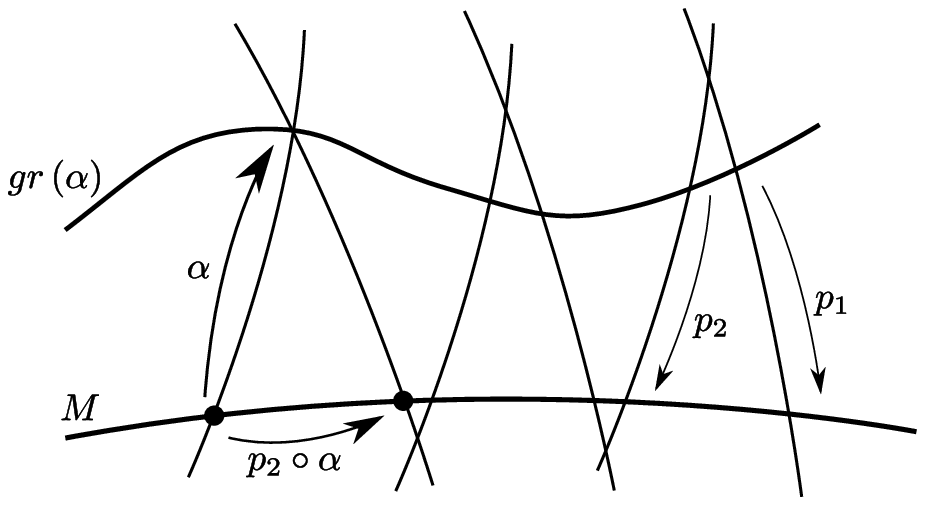}
\par\end{center}%
\end{minipage}\end{lem}
\begin{proof}
First of all, the map $A:\alpha\mapsto p_{2}\circ\alpha$ is smooth
from $\Gamma\left(p_{1}\right)$ to $C^{\infty}\left(L,L\right)$.
In particular it is continuous and therefore the set $V_{1}=A^{-1}\left(\mbox{Diff}\left(L\right)\right)$
is open. It is also non-empty since $A\left(\iota_{L}\right)=\mathbb{I}_{L}$.
On this neighbourhood of $\iota_{L}$, the map $\Psi$ is therefore
well-defined. It is also smooth, since taking the inverse of a diffeomorphism
is a Fréchet smooth map \cite[p. 92]{hamilton}. Moreover, for any
$\alpha\in V_{1}$ the image $\Psi\left(\alpha\right)$ is indeed
in $\Gamma\left(p_{2}\right)$ since $p_{2}\circ\alpha\circ\left(p_{2}\circ\alpha\right)^{-1}=\mathbb{I}_{L}$.
Now, we check that the above given formula for the inverse $\Psi^{-1}$
is correct. Indeed,\begin{eqnarray*}
\Psi^{-1}\left(\Psi\left(\alpha\right)\right) & = & \alpha\circ\left(p_{2}\circ\alpha\right)^{-1}\circ\left(p_{1}\circ\alpha\circ\left(p_{2}\circ\alpha\right)^{-1}\right)^{-1}\\
 & = & \alpha\circ\left(p_{2}\circ\alpha\right)^{-1}\circ p_{2}\circ\alpha\end{eqnarray*}
which is $\alpha$. One verifies in a similar way that $\Psi\left(\Psi^{-1}\left(\beta\right)\right)$
and prove that, up to taking a smaller $V_{1}$, the map $\Psi$ indeed
defines a local diffeomorphism around $\iota_{L}$.
\end{proof}
Notice that by construction, the sections $\alpha$ and $\Psi\left(\alpha\right)$
have the same graphs. We can now use this result to study the parametrisation
$\pi\mapsto L_{\pi}$.

\begin{lem}
\label{lem-fib-subb-diffeo}The map\begin{eqnarray*}
\mbox{Fib}\left(M,B\right) & \rightarrow & \mbox{Subb}\left(M\times B\right)\\
\pi & \mapsto & L_{\pi}\end{eqnarray*}
is a Fréchet diffeomorphism onto an open subset of $\mbox{Subb}\left(M\times B\right)$.\end{lem}
\begin{proof}
For any $\pi_{0}\in\mbox{Fib}\left(M,B\right)$, consider the graph
$L_{\pi_{0}}$ of $\hat{\pi}_{0}$ as explained at the beginning of
this section. There exists a tubular neighbourhood $N\subset M\times B$
of $L_{\pi_{0}}$ whose projection $q:N\rightarrow L_{\pi_{0}}$ is
{}``horizontal'', i.e. tangent to the fibres $M\times\left\{ b\right\} $,
$b\in B$. through the natural identification $\hat{\pi}_{0}:M\overset{\cong}{\rightarrow}L_{\pi_{0}}$,
one can view $N$ as a fibration over $M$. Denote by $p_{2}:N\rightarrow M$
the corresponding projection. On the other hand, we have of course
the {}``vertical'' projection $p_{1}:N\rightarrow M$ coming from
the projection on the first factor in $M\times B$. Finally, both
projections have a common section, precisely $\iota_{L_{\pi_{0}}}=\hat{\pi}_{0}$,
and one can thus apply Lemma \ref{lem-bi-fibration} which tells us
that the Fréchet manifolds of sections $\Gamma\left(p_{1}\right)$
and $\Gamma\left(p_{2}\right)$ are locally diffeomorphic. But $\Gamma\left(p_{1}\right)$
is just the open subset of $\mbox{Fib}\left(M,B\right)$, composed
of those fibrations $\pi:M\rightarrow B$ whose corresponding sections
$\hat{\pi}$ have their image lying in $N$. On the other hand, $\Gamma\left(p_{2}\right)$
is a suitable Fréchet chart for $\mbox{Subb}\left(M\times B\right)$
around $L_{\pi_{0}}$, as explained in Section \ref{sub-manifold-structure-Subb}.
But one can see easily that the map $\pi\mapsto L_{\pi}$ corresponds
in this chart precisely to the Fréchet diffeomorphism $\Psi$ from
Lemma \ref{lem-bi-fibration}, because $L_{\pi}$ is the graph of
$\hat{\pi}$. Therefore, we have proved that the map $\pi\mapsto L_{\pi}$
is a Fréchet local diffeomorphism. But it is also clearly injective,
therefore it is a diffeomorphism onto its image.
\end{proof}
This lemma together with Theorem \ref{thm-lag-subbundle-sous-varioche}
and Lemma \ref{lem-equiv-subbundle-fibration} finally imply directly
the following theorem.
\begin{thm}
\label{thm-fib_lag-submanifold-of-fib}$\mbox{Fib}_{\mbox{Lag}}\left(M,B\right)$
is a Fréchet submanifold of $\mbox{Fib}\left(M,B\right)$.
\end{thm}
As mentioned in the introduction, $\mbox{Fib}_{\mbox{Lag}}\left(M,B\right)$
is actually a Fréchet principal bundle with structure group $\mbox{Diff}\left(B\right)$
acting by left composition of submersions $\pi:M\rightarrow B$. This
result and its application to Hamiltonian completely integrable systems
will be presented in subsequent papers \cite{humiliere_roy_1,humiliere_roy_2}.

\subsection*{Acknowledgements}

The final writing step of this article was done in the \emph{Mathematisches
Institut Oberwolfach} during a {}``Research in Pair'' program. I
wholeheartedly thank the MFO staff who provides the mathematicians
with perfect working conditions. 

I am also very indepted to V. Humilière for his numerous useful comments
and discussions during the whole development of this project, in its
infancy and up to the final reading.

\bibliographystyle{plain}
\bibliography{/home/roy/Documents/maths/biblio/biblio_nico}

\end{document}